\documentclass[]{amsart}
\usepackage{amsmath}
\usepackage{amsthm}
\usepackage{amssymb}
\usepackage{amscd}
\usepackage{amsfonts}
\usepackage{amsbsy}
\usepackage{epsfig,afterpage}

\newtheorem {theorem} {Theorem}
\newtheorem {proposition} [theorem]{Proposition}
\newtheorem {corollary} [theorem]{Corollary}

\newcommand{\R}{\mathbb{R}}

\def \e {\varepsilon}

\usepackage[usenames,dvipsnames]{color}

\begin{document}

\title[Flow Curvature Manifold and Energy of Generalized Li\'{e}nard Systems]
{Flow Curvature Manifold and Energy\\ of Generalized Li\'{e}nard Systems}

\author[J.M. Ginoux, D. Lebiedz, J. Llibre]
{Jean-Marc Ginoux$^1$, Dirk Lebiedz$^2$ and Jaume Llibre$^3$}

\address{$^1$ Aix Marseille Univ, Universit\'{e} de Toulon, CNRS, CPT, Marseille, France}
\email{ginoux@univ-tln.fr}

\address{$^2$ Institute of Numerical Mathematics, Ulm, Germany}
\email{dirk.lebiedz@uni-ulm.de}

\address{$^3$ Departament de Matem\`{a}tiques,
Universitat Aut\`{o}noma de Barcelona, 08193 Bellaterra, Barcelona,
Spain} \email{jllibre@mat.uab.cat}



\begin{abstract}
In his famous book entitled \textit{Theory of Oscillations}, Nicolas Minorsky wrote: ``\textit{each time the system absorbs energy the curvature of its trajectory decreases} and \textit{vice versa}''. According to the \textit{Flow Curvature Method}, the location of the points where the \textit{curvature of trajectory curve}, integral of such planar \textit{singularly dynamical systems}, vanishes directly provides a first order approximation in $\varepsilon$ of its \textit{slow invariant manifold} equation. By using this method, we prove that, in the $\varepsilon$-vicinity of the \textit{slow invariant manifold} of generalized Li\'{e}nard systems, the \textit{curvature of trajectory curve} increases while the \textit{energy} of such systems decreases. Hence, we prove Minorsky's statement for the generalized Li\'{e}nard systems. Then, we establish a relationship between \textit{curvature} and \textit{energy} for such systems. These results are then exemplified with the classical Van der Pol and generalized Li\'{e}nard \textit{singularly perturbed systems}.
\end{abstract}

\maketitle

\section{Introduction}
\label{sec1}

At the end of the 1930s, a general equation of \textit{self-sustained oscillations} (\ref{eq1}) was stated by the French engineer Alfred Li\'{e}nard \cite{Lienard}. It encompassed the prototypical equation of the Dutch physicist Balthasar Van der Pol \cite{Vdp1926} modelling the so-called \textit{relaxation oscillations}\footnote{For more details see J.-M. Ginoux \cite{Ginoux2017}.}.

\begin{equation}
\label{eq1}
\dfrac{d^2x}{dt^2} + \omega f\left( x \right) \dfrac{dx}{dt} + \omega^2 x = 0.
\end{equation}

Less than fifteen years later, a more general form was provided by the American mathematicians Norman Levinson and his former student Oliver K. Smith \cite{LeviSmith}:

\begin{equation}
\label{eq2}
\dfrac{d^2x}{dt^2} + \mu f\left( x \right) \dfrac{dx}{dt} + g\left( x \right) = 0.
\end{equation}

At that time, the classical geometric theory of differential equations developed originally by Andronov \cite{Andro1}, Tikhonov \cite{Tikh} and
Levinson \cite{Lev} stated that \textit{singularly perturbed systems} possess \textit{invariant manifolds} on which trajectories evolve slowly, and toward which nearby orbits contract exponentially in time (either forward or backward) in the normal directions. These manifolds have been called asymptotically stable (or unstable) \textit{slow invariant manifolds}\footnote{In other articles the {\it slow manifold} is the approximation of order $O(\e)$ of the {\it slow invariant manifold}.}. Then, Fenichel \cite{Fen5,Fen6,Fen7,Fen8} theory\footnote{The theory of invariant manifolds for an ordinary  differential equation is based on the work of Hirsch, \textit{et al.} \cite{Hirsch}} for the \textit{persistence of normally hyperbolic invariant manifolds} enabled to establish the \textit{local invariance} of \textit{slow invariant manifolds} that possess both expanding and contracting directions and which were labeled \textit{slow invariant manifolds}.

During the last century, various methods have been developed to compute the \textit{slow invariant manifold} or, at least an asymptotic expansion in power of $\e$. The seminal works of Wasow \cite{Wasow}, Cole \cite{Cole}, O'Malley \cite{Malley1,Malley2} and Fenichel \cite{Fen5,Fen6,Fen7,Fen8} to name but a few, gave rise to the so-called \textit{Geometric Singular Perturbation Method}. According to this theory, existence as well as local invariance of the \textit{slow invariant manifold} of \textit{singularly perturbed systems} has been stated. Then, the determination of the \textit{slow invariant manifold} equation turned into a regular perturbation problem in which one generally expected the asymptotic validity of such expansion to breakdown \cite{Malley2}. Fifteen years ago, a new approach of $n$-dimensional singularly perturbed dynamical systems of ordinary differential equations with two time scales, called \textit{Flow Curvature Method} has been developed \cite{Gin}. In dimension two, it consists in considering the \textit{trajectory curves} integral of such systems as \textit{plane} curves. Based on the use of local metric properties of \textit{curvature} resulting from \textit{Differential Geometry}, this method which does not require the use of asymptotic expansions, states that the location of the points where the local \textit{curvature} of \textit{trajectory curves} of such systems, vanishes, directly provides a first order approximation in $\varepsilon$ of the \textit{slow invariant manifold} equation associated with such two-dimensional or planar \textit{singularly perturbed systems}. This method gives an implicit non intrinsic equation, because it depends on the euclidean metric. A 'kinetic energy metric' has been introduced in \cite{LebReSie} for chemical kinetic systems and an extremum principle for computing \textit{slow invariant manifolds} has been formulated \cite{LebSieUn,LebSie} which can be viewed as minimum curvature geodesics. In \cite{HeiLeb} a curvature-based differential geometry formulation for the \textit{slow manifold} problem has been used for the purpose of a coordinate-independent formulation of the invariance equation.

In his famous book entitled \textit{Theory of Oscillations}, the Russian mathematician Nicolas Minorsky \cite{Minorsky1967} wrote:

\smallskip
\begin{quote}
``\textit{each time the system absorbs energy the curvature of its trajectory decreases} and \textit{vice versa} when the energy is supplied by the system (e. g. braking) the curvature increases.''
\end{quote}
\smallskip

Thus, according to Minorsky, \textit{energy} and \textit{curvature of the trajectory} are linked by a relationship that he unfortunately didn't give. So, the aim of this work is to prove this statement in the $\varepsilon$-vicinity of the \textit{slow invariant manifold} of generalized Li\'{e}nard systems and to establish this relationship for such systems.
The paper is organized as follows. In section 2, we briefly present the definitions of \textit{singularly perturbed systems}. Then, we prove that generalized Li\'{e}nard systems are planar \textit{singularly perturbed systems}. In section 3, we recall Li\'{e}nard's assumptions for which the generalized Li\'{e}nard systems has a \textit{unique stable limit cycle} and so, a \textit{slow invariant manifold}. In section 4, we recall the main features of the \textit{Flow Curvature Method} according to which the \textit{curvature of the trajectory curve}, integral of planar \textit{singularly perturbed systems} defines a \textit{flow curvature manifold} and we state that the location of the points where this manifold vanishes directly provides a first order approximation in $\varepsilon$ of its \textit{slow invariant manifold} equation. Then, we prove  Minorsky's statement for the generalized Li\'{e}nard systems and establish a relationship between \textit{curvature} and \textit{energy} for such systems. In section 5, we exemplify these results with the classical Van der Pol \textit{singularly perturbed system}. Discussion and perspectives are presented in section 6.

\section{Singularly perturbed systems}
\label{sec2}

Thus, according to Tikhonov \cite{Tikh}, Takens \cite{Tak}, Jones \cite{Jones} and Kaper \cite{Kaper} \textit{singularly perturbed systems} may be defined such as:

\begin{equation}
\label{eq3}
\begin{array}{*{20}c}
 {{\vec {x}}' = \vec {f}\left( {\vec {x},\vec {y},\varepsilon }
\right),\mbox{ }} \hfill \\
 {{\vec {y}}' = \varepsilon \vec {g}\left( {\vec {x},\vec {y},\varepsilon }
\right)}. \hfill \\
\end{array}
\end{equation}

where $\vec {x} \in \mathbb{R}^m$, $\vec {y} \in \mathbb{R}^p$, $\varepsilon \in \mathbb{R}^ + $, and the prime denotes differentiation with respect to the independent variable $t$. The functions $\vec {f}$ and $\vec {g}$ are assumed to be $C^\infty$ functions\footnote{In certain applications these functions will be
supposed to be $C^r$, $r \geqslant 1$.} of $\vec {x}$, $\vec {y}$ and $\varepsilon$ in $U\times I$, where $U$ is an open subset of
$\mathbb{R}^m\times \mathbb{R}^p$ and $I$ is an open interval containing $\varepsilon = 0$.

\smallskip

In the case when $0 < \varepsilon \ll 1$, i.e., $\e$ is a small positive number, the variable $\vec {x}$ is called \textit{fast} variable, and $\vec {y}$ is called \textit{slow} variable. Using Landau's notation: $O\left( {\varepsilon^k} \right)$ represents a function $f$ of $x$ and $\varepsilon $ such that $f(u,\e)/\e^k$ is bounded for positive $\e$  going to zero, uniformly for $u$ in the given domain. It is used to consider that generally $\vec {x}$ evolves at an $O\left( 1 \right)$ rate; while $\vec {y}$ evolves at an $O\left( \varepsilon \right)$ \textit{slow} rate. Reformulating system (\ref{eq3}) in terms of the rescaled variable $\tau = \varepsilon t$, we obtain

\begin{equation}
\label{eq4}
\begin{aligned}
\varepsilon \dot {\vec {x}} & = \vec{f} \left( {\vec{x},\vec{y}, \varepsilon} \right), \\
\dot {\vec {y}} & = \vec {g}\left( {\vec{x}, \vec{y},\varepsilon }
\right).
\end{aligned}
\end{equation}

The dot represents the derivative with respect to the new independent variable $\tau $.

\smallskip

The independent variables $t$ and $\tau $ are referred to the \textit{fast} and \textit{slow} times, respectively, and (\ref{eq3}) and (\ref{eq4}) are called the \textit{fast} and \textit{slow} systems, respectively. These systems are equivalent whenever $\varepsilon \ne 0$, and they are labeled \textit{singular perturbation problems} when $0 < \varepsilon \ll 1$. The label ``singular'' stems in part from the discontinuous limiting behavior in system (\ref{eq3}) as $\varepsilon \to 0$.

\smallskip

In such case system (\ref{eq4}) leads to a differential-algebraic system called \textit{reduced slow system} whose dimension decreases from $m + p = n$ to $p$. Then, the \textit{slow} variable $\vec {y} \in \mathbb{R}^p$ partially evolves in the submanifold $M_0$ called the \textit{critical manifold}\footnote{It corresponds to the approximation of the slow invariant manifold, with an error of $O(\e)$.} and defined by

\begin{equation}
\label{eq5} M_0 := \left\{ {\left( {\vec {x},\vec {y}} \right):\vec
{f}\left( {\vec {x},\vec {y},0} \right) = {\vec {0}}} \right\}.
\end{equation}

When $D_xf$ is invertible, thanks to Implicit Function Theorem, $M_0 $ is given by the graph of a $C^\infty $ function $\vec {y} = \vec {F}_0 \left( \vec {x} \right)$ for $\vec {x} \in D$, where $D\subseteq \mathbb{R}^p$ is a compact, simply connected domain and the boundary of $D$ is an $(p - 1)$--dimensional $C^\infty$ submanifold\footnote{The set D is overflowing invariant with respect to (\ref{eq4}) when $\varepsilon = 0$.}.

\smallskip

According to Fenichel theory \cite{Fen5, Fen6, Fen7, Fen8} if $0 < \varepsilon \ll 1$ is sufficiently small, then there exists a function $\vec {F}\left( {\vec {x},\varepsilon } \right)$ defined on D such that the manifold

\begin{equation}
\label{eq6} M_\varepsilon := \left\{ {\left( {\vec {x},\vec {y}}
\right):\vec {y} = \vec {F}\left( {\vec {x},\varepsilon } \right)}
\right\},
\end{equation}

is locally invariant under the flow of system (\ref{eq3}). Moreover, there exist perturbed local stable (or attracting) $M_a$ and unstable (or repelling) $M_r$ branches of the \textit{slow invariant manifold} $M_\varepsilon$. Thus, normal hyperbolicity of $M_\e$ is lost via a saddle-node bifurcation of the \textit{reduced slow system} (\ref{eq4}).

In dimension two, planar \textit{singularly perturbed dynamical systems} (\ref{eq4}) for which with $\vec {x} \in \mathbb{R}^1$, $\vec {y} \in \mathbb{R}^{1}$, i.e. $(m,p)=(1,1)$ read:

\begin{equation}
\label{eq7}
\begin{aligned}
\varepsilon \dot {x} & = f \left( x, y , \varepsilon \right), \\
          \dot { y } & = g \left( x, y ,\varepsilon \right).
\end{aligned}
\end{equation}

\section{Generalized Li\'{e}nard systems}
\label{sec3}

Starting from the generalized Li\'{e}nard equation (\ref{eq2}) which is a paradigm for \textit{self-sustained oscillations} and by posing: $t \to \mu t$ and $\mu = 1 / \sqrt{\e}$, we have:

\begin{equation}
\label{eq8}
\begin{aligned}
\varepsilon \dot {x} & = y - F\left( x \right), \\
            \dot {y} & = - g \left( x \right).
\end{aligned}
\end{equation}

It is thus obvious that generalized Li\'{e}nard system (\ref{eq8}) is a \textit{singularly perturbed dynamical systems} (\ref{eq4}) for which with $\vec {x} \in \mathbb{R}^1$, $\vec {y} \in \mathbb{R}^{1}$, i.e. $(m,p)=(1,1)$, i.e. planar \textit{singularly perturbed system}.\\

According to Lefschetz \cite{Lefschetz}, under the following assumptions:\\

\begin{itemize}
\item[I.] $f\left(x\right)$ is even, $g\left(x\right)$ is odd, $x g\left(x\right) > 0$ for all $x \neq 0$; $f\left(0\right) < 0$;
\item[II.] $f\left(x\right)$ and $g\left(x\right)$ are continuous for all $x$; $g\left(x\right)$ satisfies Lipschitz condition for all $x$;
\item[III.] $F\left(x\right) \to \pm \infty$ with $x$;
\item[IV.] $F\left(x\right)$ has a single positive zero $x = a$ and is monotone increasing for $x \geqslant a$,
\end{itemize}

the generalized Li\'{e}nard equation (\ref{eq2}), as well as the generalized Li\'{e}nard system (\ref{eq8}) has a \textit{unique stable limit cycle} and so, possesses a \textit{slow invariant manifold}.

\section{Flow Curvature Method}
\label{sec4}

Fifteen years ago, a new approach called \textit{Flow Curvature Method} and based on the use of \textit{Differential Geometry} properties of \textit{curvatures} has been developed by Ginoux \textit{et al.} \cite{GiRo1, GiRo2,Gin}. According to this method, the \textit{curvature of the flow} of \textit{trajectory curve} integral of any $n$-dimensional dynamical system defines a \textit{manifold} associated with this system and called \textit{flow curvature manifold}. In the case of $n$-dimensional \textit{singularly perturbed system} (\ref{eq4}) for which with $\vec {x} \in \mathbb{R}^1$, $\vec {y} \in \mathbb{R}^{n-1}$, i.e. $(m,p)=(1,n-1)$, it has been stated by Ginoux \textit{et al.} \cite{GiRo1, GiRo2,Gin} that the location of the points where this \textit{flow curvature manifold} vanishes directly provides a $(n - 1)$-order approximation in $\e$ of its \textit{slow manifold}, the \textit{invariance} of which is stated according to Darboux theorem \cite{Darboux}. The cases of three and four-dimensional \textit{singularly perturbed dynamical system} (\ref{eq4}) for which $(m,p)=(1,1)$, $(m,p)=(2,1)$, $(m,p)=(3,1)$, $(m,p)=(2,2)$ and $(m,p)=(2,3)$ have been also analyzed by Ginoux \textit{et al.} \cite{GinLi1, GinLi2, GinLi3,GinLi4,GinLi5,GinLi6}.

In the case of two-dimensional or planar \textit{singularly perturbed dynamical systems} (\ref{eq7}) for which with $\vec {x} \in \mathbb{R}^1$, $\vec {y} \in \mathbb{R}^{1}$, i.e. $(m,p)=(1,1)$ we have the following result.

\begin{proposition}
\label{prop1} The location of the points where the \textit{curvature of the flow}, i.e. the \textit{curvature of the trajectory curve} $\vec {X} = \left( x, y \right)$, integral of any two-dimensional or planar \textit{singularly perturbed system} \eqref{eq7} vanishes directly provides a first order approximation in $\varepsilon$ of its its one-dimensional \textit{slow manifold} $M_{\varepsilon}$, the equation of which reads

\begin{equation}
\label{eq9}
\phi ( {\vec {X}, \varepsilon} ) = \det(\ddot {\vec {X}}, \dot {\vec {X}}) = 0
\end{equation}

where $\dot {\vec {X}}$ and $\ddot {\vec {X}}$ represent the time derivatives of $\vec {X}  = (x, y)^t$.
\end{proposition}

\begin{proof}
For proof of this proposition see \cite[p. 185 and next]{Gin}.
\end{proof}

\subsection{Invariance}

According to Schlomiuk \cite{Schlomiuk} and Llibre \textit{et al.} \cite{LLibreMedrado} the concept of \textit{invariant manifold} has
been originally introduced by Gaston Darboux \cite[p. 71]{Darboux} in a memoir entitled: \textit{Sur les \'{e}quations diff\'{e}rentielles
alg\'{e}briques du premier ordre et du premier degr\'{e}} and can be stated as follows.

\begin{proposition}
\label{prop2}
The \textit{manifold} defined by $\phi ( \vec {X}, \varepsilon) = 0$ where $\phi $ is a $C^1$ in an open set U, is \textit{invariant}
with respect to the flow of \eqref{eq7} if there exists a $C^1$ function denoted by $\kappa ( \vec {X}, \varepsilon)$ and called
cofactor which satisfies

\begin{equation}
\label{eq10}
L_{\overrightarrow V } \phi ( \vec {X}, \varepsilon) = \kappa( \vec {X}, \varepsilon) \phi ( \vec {X}, \varepsilon),
\end{equation}

for all $\vec {X} \in U$, and with the Lie derivative operator defined as

\[
L_{\overrightarrow V } \phi = \overrightarrow V \cdot \overrightarrow \nabla \phi = \sum\limits_{i = 1}^n {\frac{\partial
\phi }{\partial x_i }\dot {x}_i } = \frac{d\phi }{dt}.
\]

\end{proposition}

\begin{proof}
For a proof of this proposition see \cite[p. 187 and next]{Gin} and below.
\end{proof}

\begin{corollary}
For any two-dimensional or planar \textit{singularly perturbed dynamical system} \eqref{eq7}, we have:

\begin{equation}
\label{eq11}
L_{\overrightarrow V } \phi = \dfrac{d\phi }{dt} = {\rm Tr} \left( J \right) \phi + \det( \dfrac{dJ}{dt} \dot{\vec{X}}, \dot{\vec{X}} ).
\end{equation}

where $J$ is the Jacobian matrix of the vector field \eqref{eq7}.

\end{corollary}

\begin{proof}
For any $n$-dimensional dynamical system as well as any $n$-dimensional \textit{singularly perturbed system} (\ref{eq4}), it is easy to prove that:

\begin{equation}
\label{eq12}
\ddot{\vec{X}} = J \dot{\vec{X}}.
\end{equation}

The time derivative of this equation (\ref{eq12}) gives:

\begin{equation}
\label{eq13}
\dddot{\vec{X}} = J \ddot{\vec{X}} + \dfrac{dJ}{dt} \dot{\vec{X}}.
\end{equation}

Moreover, let's notice that for two-dimensional dynamical systems, the \textit{slow invariant manifold} (\ref{eq9}) can be also written as:

\begin{equation}
\label{eq14}
\phi ( {\vec {X}, \varepsilon} ) = \det( \ddot {\vec {X}}, \dot {\vec {X}}) = ( \ddot {\vec {X}} \wedge \dot{\vec {X}} ) \cdot  \vec{k} =  0
\end{equation}

where $\vec{k}$ is the unit vector of the $z$-axis. Time derivative of (\ref{eq14}) provides:

\begin{equation}
\label{eq15}
\frac{d\phi }{dt} =  ( \dddot {\vec {X}} \wedge \dot {\vec {X}} ) \cdot \vec{k}.
\end{equation}

By replacing (\ref{eq13}) in (\ref{eq15}), we obtain:

\begin{equation}
\label{eq16}
\frac{d\phi }{dt} = \det( \dddot {\vec {X}},\dot {\vec {X}}) = ( J \ddot {\vec {X}} \wedge \dot {\vec {X}} ) \cdot \vec{k} + ( \dfrac{dJ}{dt} \dot{\vec{X}} \wedge \dot{\vec{X}} ) \cdot \vec{k}.
\end{equation}

By using the following well-known identity:

\begin{equation}
\label{eq17}
J \vec{a} \wedge \vec{b} + \vec{a} \wedge J \vec{b} = {\rm Tr}(J) (\vec{a} \wedge \vec{b})
\end{equation}

where $J$ is the Jacobian matrix, we find that:

\begin{equation}
\label{eq18}
\ddot {\vec {X}} \wedge J \dot {\vec {X}} + J \ddot {\vec {X}} \wedge \dot {\vec {X}} = {\rm Tr}(J) ( \ddot {\vec {X}} \wedge \dot {\vec {X}})
\end{equation}

Finally, we obtain:

\begin{equation}
\label{eq19}
\dfrac{d\phi }{dt} = {\rm Tr}(J) ( \ddot {\vec {X}} \wedge \dot {\vec {X}} ) \cdot \vec{k} + ( \dfrac{dJ}{dt} \dot{\vec{X}} \wedge  \dot{\vec{X}} ) \cdot \vec{k} = {\rm Tr} \left( J \right) \phi + \det( \dfrac{dJ}{dt} \dot{\vec{X}}, \dot{\vec{X}} ).
\end{equation}

\end{proof}

\section{Minorsky's statement}

In order to establish Minorsky's statement for the generalized Li\'{e}nard system \eqref{eq8}, we introduce the following propositions.

\begin{proposition}
\label{prop3}
In the $\varepsilon$-vicinity of the \textit{slow} part of the \textit{critical manifold}, the \textit{slow invariant manifold} \eqref{eq9} of the generalized Li\'{e}nard system \eqref{eq8} is positive provided that $g'(x) \geqslant 0$ and under the previous assumptions $(I - IV)$.
\end{proposition}

\begin{proof}

According to the previous assumptions I. \& II. $g(x)$ is odd and continuous, so we have $g(0) = 0$. Thus, at the crossings with the $y$-axis the tangents to the \textit{trajectory curves} are horizontal ($\dot{y}/\dot{x} = 0$ since $g(0) = 0$), and at the crossings with the curve $y = F(x)$ they are vertical ($\dot{y}/\dot{x} = \infty$ since $y - F(x) = 0$). Moreover, since this gradient ($\dot{y}/\dot{x}$) is negative in the $\varepsilon$-vicinity of the \textit{slow} part of the \textit{critical manifold}, the \textit{trajectory curve} cannot leave its neighborhood and any tendency for it to move away from it would be counteracted by a rapid growth in magnitude of this negative gradient. Then, according to Lefschetz \cite{Lefschetz}:\\

``We see from (\ref{eq8}) that with increasing time:\\

\begin{itemize}
\item $x(t)$ increases above the \textit{critical manifold},\\
\item $x(t)$ decreases below the \textit{critical manifold},\\
\item $y(t)$ increases to the left of the $y$ axis,\\
\item $y(t)$ decreases to the right of the $y$ axis.''\\
\end{itemize}

Due to the symmetries of the generalized Li\'{e}nard system, we will only consider the right half part of the $xy$-plane to state this proposition.

Thus, we deduce from what precedes that below the \textit{slow} part of the \textit{critical manifold}, $x(t)$ decreases and $y(t)$ decreases. Let's remind that the \textit{trajectory curve} is below the \textit{slow} part of the \textit{critical manifold}.

So, if in the $\varepsilon$-vicinity of the \textit{slow} part of the \textit{critical manifold}, $x(t)$ decreases and $y(t)$ decreases, it follows that we have:
\begin{equation}
\label{eq20}
\begin{aligned}
& \dot{x}(t) < 0\\
& \dot{y}(t) < 0.
\end{aligned}
\end{equation}

Then, by applying these results to the generalized Li\'{e}nard system (\ref{eq8}) and while using assumptions (I - IV), we can prove that in the $\varepsilon$-vicinity of the \textit{slow} part of the \textit{critical manifold}, we have:

\begin{equation}
\label{eq21}
\begin{aligned}
& \ddot{x}(t) < 0\\
& \ddot{y}(t) > 0 \quad \mbox{if} \quad g'(x) \geqslant 0.
\end{aligned}
\end{equation}

If $x(t)$ decreases and $y(t)$ decreases and since $F(x)$ is monotone increasing for $x \geqslant a$, $-F(x)$ is monotone decreasing. Thus, $\dot{x}(t)$ decreases as time increases and so $\ddot{x}(t) < 0$. Since $\dot{y}(t) = - g(x)$, we have: $\ddot{y}(t) = - g'(x)\dot{x}(t)$. So, if $g'(x) \geqslant 0$, $\ddot{y}(t) > 0$. By considering that $\ddot{y}(t) = - g'(x)\dot{x}(t)$, the first order approximation in $\varepsilon$ of the \textit{slow invariant manifold} (\ref{eq9}) of the generalized Li\'{e}nard system (\ref{eq8}) reads:

\begin{equation}
\label{eq22}
\phi ( x, y, \varepsilon ) = \ddot{x}\dot{y} + g'(x)\dot{x}^2 = 0
\end{equation}

Thus, according to (\ref{eq20}-\ref{eq21}), it follows that $\phi ( x, y, \varepsilon ) \geqslant 0$ provided that $g'(x) \geqslant 0$.

\end{proof}

\begin{proposition}
\label{prop5}
In the $\varepsilon$-vicinity of the \textit{slow} part of the \textit{critical manifold}, the time derivative of the \textit{slow invariant manifold} \eqref{eq9} of the generalized Li\'{e}nard system \eqref{eq8} is positive provided that $g'(x) \geqslant 0$ and under the previous assumptions $(I - IV)$.
\end{proposition}

\begin{proof}
The time derivative of the \textit{slow invariant manifold} (\ref{eq9}) reads:

\begin{equation}
\label{eq23}
\dfrac{d\phi }{dt} = \dddot{x}\dot{y} - \dddot{y} \dot{x} = 0
\end{equation}

From the generalized Li\'{e}nard system (\ref{eq8}), we find that:

\begin{equation}
\label{eq24}
\ddot{x}(t) = \dfrac{1}{\e} \left( \dot{y} - \dfrac{dF(x)}{dt} \right) = \dfrac{1}{\e} \left( \dot{y} - F'(x) \dot{x} \right)
\end{equation}

Taking into account that $F'(x) = f(x)$ and $\dot{y}(t) = - g(x)$, we have:

\begin{equation}
\label{eq25}
\ddot{x}(t) = \dfrac{1}{\e} \left( - g(x) - f(x) \dot{x} \right)
\end{equation}

We have previously stated that $x(t)$, $y(t)$, $\dot{x}(t)$ and $\dot{y}(t)$ decrease in the $\varepsilon$-vicinity of the \textit{slow} part of the \textit{critical manifold}. According to assumption IV, $F(x)$ is monotone increasing for $x \geqslant a$. So, $F'(x) = f(x) > 0$. Moreover, we have supposed that $g'(x) \geqslant 0$ which implies that $g(x)$ increases. Thus, both $-g(x)$ and $- f(x) \dot{x}$ decrease. It follows that $\ddot{x}(t)$ decreases. It leads to $\dddot{x}(t) < 0$. Now, starting from the generalized Li\'{e}nard system (\ref{eq8}), we find that:

\begin{equation}
\label{eq26}
\dddot{y}(t) = - g''(x)\dot{x}^2 - g'(x)\ddot{x}
\end{equation}

By replacing this expression (\ref{eq26}) in that of the time derivative of the \textit{slow invariant manifold} (\ref{eq23}), we obtain:

\begin{equation}
\label{eq27}
\dfrac{d\phi }{dt} = \dddot{x}\dot{y} + \dot{x} \left( g''(x)\dot{x}^2 + g'(x)\ddot{x} \right).
\end{equation}

Let's notice that:

\[
g''(x)\dot{x}^2 + g'(x)\ddot{x} = \dfrac{d}{dt} \left( g'(x) \dot{x} \right).
\]

Since we have stated that $\ddot{x}(t) < 0$, it implies that $\dot{x}(t) < 0$ decreases. Hence, provided that $g'(x) \geqslant 0$, $g'(x) \dot{x}$ decreases also. As a consequence, $\dfrac{d}{dt} \left( g'(x) \dot{x} \right) \leqslant 0$. Thus, since we have stated that $\dot{x}(t) < 0$, $\ddot{x}(t) < 0$ and $\dddot{x}(t) < 0$ all terms of $d\phi/dt$ are thus positive provided that $g'(x) \geqslant 0$.

\end{proof}

Now, let's prove Minorsky's statement for the generalized Li\'{e}nard system (\ref{eq8}).

\begin{proposition}
\label{prop6}
``\textit{each time the system absorbs energy the curvature of its trajectory decreases} and \textit{vice versa} when the energy is supplied by the system (e. g. braking) the curvature increases.''
\end{proposition}

\begin{proof}

According to Berg\'{e} \textit{et al.} \cite{Berge} a classical way to express the variation of energy according to time in generalized Li\'{e}nard equation (\ref{eq2}) (in which we have posed $t \to \mu t$ and $\mu = 1 / \sqrt{\e}$) is to multiply this equation by $\dot{x}(t)$. By doing that, we obtain:

\begin{equation}
\label{eq28}
\e \dot{x}\ddot{x} + f\left( x \right) \dot{x}^2 + g\left( x \right)\dot{x} = 0.
\end{equation}

By taking $G'(x) = g(x)$, we find that:

\begin{equation}
\label{eq29}
\dfrac{d}{dt}\left( \e \dfrac{\dot{x}^2}{2} + G(x) \right) = - f(x) \dot{x}^2
\end{equation}

But, from assumption IV, it follows that $F(x)$ is monotone increasing for $x \geqslant a$. So, $F'(x) = f(x) > 0$. Let's consider that the \textit{energy} reads:

\begin{equation}
\label{eq30}
E = \dfrac{\e \dot{x}^2}{2} + G(x),
\end{equation}

where, according to Lefschetz \cite{Lefschetz}, ``in the ``spring'' interpretation $ \e \dot{x}^2/2$ is the kinetic energy and $G(x)$ is the potential energy\footnote{In fact Lefschetz \cite{Lefschetz} provided for the energy $E$ a different expression from the previous one (\ref{eq30}). However, it will established in the Appendix that they are exactly the same.}''. So, we have:

\begin{equation}
\label{eq31}
\dfrac{dE}{dt} = - f(x) \dot{x}^2 < 0.
\end{equation}

Thus, in the $\varepsilon$-vicinity of the \textit{slow} part of the \textit{critical manifold}, the \textit{curvature} $\phi(x, y, \e)$ increases while the energy $E$ of the generalized Li\'{e}nard system (\ref{eq8}) decreases as claimed by Minorsky.

\end{proof}

Now, let's establish a relationship between the \textit{flow curvature manifold} \eqref{eq9} and the \textit{energy} \eqref{eq30} of the generalized Li\'{e}nard system \eqref{eq8}. According to \eqref{eq9}, the \textit{flow curvature manifold} \eqref{eq22} of the generalized Li\'{e}nard system \eqref{eq8} reads:

\[
\phi ( x, y, \varepsilon ) = \ddot{x}\dot{y} + g'(x)\dot{x}^2 = 0
\]

By taking the derivative of the first equation of \eqref{eq8} and replacing into the previous one, we find:

\begin{equation}
\label{eq32}
\phi ( x, y, \varepsilon ) = \dfrac{1}{\e} \left( \dot{y} - f(x) \dot{x} \right) \dot{y} + g'(x)\dot{x}^2 = 0,
\end{equation}
which can be also written as:

\begin{equation}
\label{eq33}
\e \phi ( x, y, \varepsilon ) = \e g'(x)\dot{x}^2 +  \dot{y}^2  - f(x) \dot{x}\dot{y}.
\end{equation}

By using \eqref{eq30}, the \textit{flow curvature manifold} reads:

\begin{equation}
\label{eq34}
\e \phi ( x, y, \varepsilon ) = 2 g'(x) E +  \dot{y}^2 - 2 g'(x)G(x)  - f(x) \dot{x}\dot{y}.
\end{equation}

Let's notice that since $x g(x) > 0$ for $x \neq 0$ from assumption I, it follows that

\[
G(x) = \int_0^x g(s) \, \mathrm{d}s > 0 \quad \Longrightarrow \quad G'(x) = g(x)
\]
As a consequence, the \textit{energy} $E >0$. Moreover, since $F(x)$ is monotone increasing for $x \geqslant a$ from assumption I, it follows that

\[
F(x) = \int_0^x f(s) \, \mathrm{d}s > 0 \quad \Longrightarrow \quad F'(x) = f(x) > 0.
\]

We have also stated previously that $\dot{x}(t) < 0$ and $\dot{y}(t) < 0$. In Proposition 4, we have made the assumption that $g'(x) \geqslant 0$. Thus, taking into account all these considerations leads to the fact that the first and last terms of the right hand side of the \textit{flow curvature manifold} \eqref{eq34} are positive. So, let's focus on the sign of $H = \dot{y}^2 - 2 g'(x)G(x)$. Since $\dot{y} = -g(x)$, $g(x) = G'(x)$, this expression reads:

\begin{equation}
\label{eq35}
H =  G'(x)^2 - 2 G(x)G''(x).
\end{equation}
To solve this ordinary differential equation \eqref{eq35}, let's pose: $Y(x) = G(x)$ and then, $U = Y'(x)$. We obtain:

\begin{equation}
\label{eq36}
H =  U^2 - 2 Y U \dfrac{dU}{dY}.
\end{equation}
By posing $Z = U^2$, we have:

\begin{equation}
\label{eq37}
H =  Z - Y \dfrac{dZ}{dY}.
\end{equation}

The solution of $H = 0$ leads to $Z = C_1 Y$. But since we have posed $Z = U^2 = Y'(x)^2$, we have:

\begin{equation}
\label{eq38}
Y'(x)^2 = C_1 Y(x) \quad \Longrightarrow \quad \dfrac{dY}{dx} = \pm \sqrt{C_1 Y(x)}.
\end{equation}
Taking into account that we have posed $Y(x) = G(x)$, the solution of this differential equation is:

\begin{equation}
\label{eq39}
G(x) = \dfrac{C_1}{2} \left( x + C_2 \right)^2.
\end{equation}

\subsection{First case $H \geqslant 0$}
According to what precedes the conditions under which $H \geqslant 0$ are equivalent to $Y'(x)^2 \leqslant C_1 Y(x)$. So, we have:

\begin{equation}
\label{eq40}
\left(Y'(x) - \sqrt{C_1 Y(x)} \right)\left(Y'(x) + \sqrt{C_1 Y(x)} \right) \leqslant 0
\end{equation}
This inequality leads to two subcases. Either the first term is negative and the second positive or the reverse. The first subcase corresponds to:

\begin{equation}
\label{eq41}
\begin{aligned}
& Y'(x) \leqslant \sqrt{C_1 Y(x)} \quad \Longrightarrow \quad G(x) \leqslant \dfrac{C_1}{2} \left( x + C_2 \right)^2, \\
& Y'(x) \geqslant - \sqrt{C_1 Y(x)} \quad \Longrightarrow \quad G(x) \geqslant - \dfrac{C_1}{2} \left( x + C_2 \right)^2.
\end{aligned}
\end{equation}
However, we have posed $Y(x) =  G(x)$ and so, $Y'(x) =  g(x) >0$. Thus we have:

\begin{equation}
\label{eq42}
0 < Y'(x) \leqslant \sqrt{C_1 Y(x)} \quad \Longrightarrow \quad 0 < G(x) \leqslant \dfrac{C_1}{2} \left( x + C_2 \right)^2.
\end{equation}

The second subcase corresponds to:

\begin{equation}
\label{eq43}
\begin{aligned}
& Y'(x) \geqslant \sqrt{C_1 Y(x)} \quad \Longrightarrow \quad G(x) \geqslant \dfrac{C_1}{2} \left( x + C_2 \right)^2, \\
& Y'(x) \leqslant - \sqrt{C_1 Y(x)} \quad \Longrightarrow \quad G(x) \leqslant - \dfrac{C_1}{2} \left( x + C_2 \right)^2.
\end{aligned}
\end{equation}
This last subcase is inconsistent with the assumption $Y'(x) =  g(x) > 0$. So we must reject it. As a consequence, since $g(x) = G'(x)$ we deduce from \eqref{eq42} that:

\begin{equation}
\label{eq44}
0 < g(x) \leqslant C_1\left( x + C_2 \right).
\end{equation}

This first case which corresponds to the Van der Pol \cite{Vdp1926} and Li\'{e}nard \cite{Lienard} systems will be exemplified in the last section.

\subsection{Second case $H \leqslant 0$}
Now let's analyze the conditions under which $H \leqslant 0$ are equivalent to $Y'(x)^2 \geqslant C_1 Y(x)$. So, we have:

\begin{equation}
\label{eq45}
\left(Y'(x) - \sqrt{C_1 Y(x)} \right)\left(Y'(x) + \sqrt{C_1 Y(x)} \right) \geqslant 0
\end{equation}
This inequality leads to two subcases. Either both terms are positive or negative.

The first subcase corresponds to:

\begin{equation}
\label{eq46}
\begin{aligned}
& Y'(x) \geqslant \sqrt{C_1 Y(x)} \quad \Longrightarrow \quad G(x) \geqslant \dfrac{C_1}{2} \left( x + C_2 \right)^2, \\
& Y'(x) \geqslant - \sqrt{C_1 Y(x)} \quad \Longrightarrow \quad G(x) \leqslant - \dfrac{C_1}{2} \left( x + C_2 \right)^2.
\end{aligned}
\end{equation}
Since we have posed $Y(x) =  G(x)$ and so, $Y'(x) =  g(x) >0$. Thus we have:

\begin{equation}
\label{eq47}
Y'(x) \geqslant \sqrt{C_1 Y(x)} \quad \Longrightarrow \quad G(x) \geqslant \dfrac{C_1}{2} \left( x + C_2 \right)^2.
\end{equation}

The second subcase corresponds to:

\begin{equation}
\label{eq48}
\begin{aligned}
& Y'(x) \leqslant \sqrt{C_1 Y(x)} \quad \Longrightarrow \quad G(x) \leqslant \dfrac{C_1}{2} \left( x + C_2 \right)^2, \\
& Y'(x) \leqslant - \sqrt{C_1 Y(x)} \quad \Longrightarrow \quad G(x) \leqslant - \dfrac{C_1}{2} \left( x + C_2 \right)^2.
\end{aligned}
\end{equation}
This last subcase is inconsistent with the assumption $Y'(x) =  g(x) > 0$. So, we must reject it. As a consequence, since $g(x) = G'(x)$, we deduce from \eqref{eq47} that:

\begin{equation}
\label{eq49}
g(x) \geqslant C_1\left( x + C_2 \right).
\end{equation}
This second case corresponds to the generalized Li\'{e}nard systems \eqref{eq8} and will be also exemplified in the last section.

\section{Curvature and energy of generalized Li\'{e}nard systems}

As previously recalled, according to Minorsky, \textit{energy} and \textit{curvature of the trajectory} are linked by a relationship that he unfortunately didn't give. In this section we establish such a relationship for the generalized Li\'{e}nard systems \eqref{eq8}. By taking: $H =  \dot{y}^2 - 2 g'(x)G(x)$, the \textit{flow curvature manifold} \eqref{eq34} reads:

\begin{equation}
\label{eq50}
\e \phi ( x, y, \varepsilon ) = 2 g'(x) E +  H  - f(x) \dot{x}\dot{y}.
\end{equation}
By taking the time derivative of this expression \eqref{eq50}, we obtain:

\begin{equation}
\label{eq51}
\e \dfrac{d \phi}{dt} = 2 g''(x) \dot{x}E + 2g'(x)\dfrac{dE}{dt} +  \dfrac{dH}{dt}  - f'(x) \dot{x}^2\dot{y} - f(x) \ddot{x}\dot{y} - f(x) \dot{x}\ddot{y}.
\end{equation}
But according to Corollary 3, we have:

\[
\dfrac{d\phi }{dt} = {\rm Tr} \left( J \right) \phi + \det( \dfrac{dJ}{dt} \dot{\vec{X}}, \dot{\vec{X}} ).
\]

For the generalized Li\'{e}nard systems \eqref{eq8}, we obtain:

\begin{equation}
\label{eq52}
\e \dfrac{d \phi}{dt} = - f'(x) \dot{x}^2\dot{y} - f(x) \ddot{x}\dot{y} + f(x) \dot{x}\ddot{y}.
\end{equation}

By equalling Eqs. (\ref{eq51}-\ref{eq52}), we have:

\begin{equation}
\label{eq53}
2 g''(x) \dot{x}E + 2g'(x)\dfrac{dE}{dt} +  \dfrac{dH}{dt}  = 2 f(x) \dot{x}\ddot{y}.
\end{equation}

Let's notice that:

\begin{equation}
\label{eq54}
g''(x) \dot{x}E + g'(x)\dfrac{dE}{dt}  = \dfrac{d}{dt}\left( g'(x)E \right).
\end{equation}
Thus Eq. \eqref{eq53} can be written as:

\begin{equation}
\label{eq55}
\dfrac{d}{dt}\left( 2 g'(x)E + H \right) = 2 f(x) \dot{x}\ddot{y}.
\end{equation}
But since $f(x) >0$, $\dot{x} < 0$ and $\ddot{y} >0$, the right hand side of Eq. \eqref{eq55} is negative. Moreover, we have assumed in Proposition 4 that $g'(x) \geqslant 0$ and stated that $E >0$. Hence, we have:

\begin{equation}
\label{eq56}
\dfrac{d}{dt}\left( 2 g'(x)E \right) < -  \dfrac{dH}{dt}.
\end{equation}
So let's focus on the sign of $dH / dt$. By replacing $H$ by its expression given by Eq. \eqref{eq35}, we find that:

\begin{equation}
\label{eq57}
\dfrac{dH}{dt} = - 2 G(x)G'''(x)\dot{x}.
\end{equation}

Since $G(x) > 0$ and $\dot{x} < 0$, we have the following two cases.

If $G'''(x) \leqslant 0$, $dH/dt \leqslant 0$, then $\dfrac{d}{dt}\left( 2 g'(x)E \right) < 0$. This corresponds to the previous first case $H \geqslant 0$ which led to $0 < g(x) \leqslant C_1\left( x + C_2 \right)$ (see Sect. 5.1). This is the case of Van der Pol \cite{Vdp1926} and Li\'{e}nard \cite{Lienard} systems that will be exemplified in the next section.

If $G'''(x) \geqslant 0$, $dH/dt \geqslant 0$, then $\dfrac{d}{dt}\left( 2 g'(x)E \right)$ is still negative but has now the positive upper bound $dH/dt$. This corresponds to the previous second case $H \leqslant 0$ which led to $g(x) \geqslant C_1\left( x + C_2 \right)$ (see Sect. 5.2). This is the case of the generalized Li\'{e}nard systems \eqref{eq8} that will be also exemplified in the next section.

\section{Applications}

In this last section we apply the results established in this work to the classical Van der Pol \cite{Vdp1926} and to the generalized Li\'{e}nard \cite{LlibreMereu} \textit{singularly perturbed systems}.

\subsection{Van der Pol \textit{singularly perturbed system}}

In his original publication of 1926, Balthasar Van der Pol \cite{Vdp1926} provided the following prototypic ordinary differential equation for modeling the relaxation oscillations:

\begin{equation}
\label{eq58}
\dfrac{d^2x}{dt^2} + \mu \left( x^2 - 1 \right) \dfrac{dx}{dt} +  x = 0.
\end{equation}
By posing: $t \to \mu t$ and $\mu = 1 / \sqrt{\e}$, such equation \eqref{eq58} can be written as:

\begin{equation}
\label{eq59}
\begin{aligned}
\varepsilon \dot {x} & = y - \left( \dfrac{x^3}{3} - x \right), \\
            \dot {y} & = - x.
\end{aligned}
\end{equation}
Thus we have: $F(x) = \dfrac{x^3}{3} - x$, $f(x) = F'(x) = x^2 - 1$, $g(x) = x$, $g'(x) = 1$ and $G(x) = \dfrac{x^2}{2} + C$ where we can take $C = 0$. According to \eqref{eq35}, we have: $H = 0$, and so $\dfrac{dH}{dt} = 0$. From \eqref{eq31} it follows that:

\begin{equation}
\label{eq60}
\dfrac{dE}{dt} = -  \left( x^2 - 1 \right) \dot{x}^2.
\end{equation}

The function $x^2 - 1 \geqslant 0$ for $x \in ] - \infty, - 1] \bigcup [+1, + \infty [$, and so $dE/dt < 0$ within this interval which contains the \textit{flow curvature manifold} that is to say, a first order approximation in $\varepsilon$ of the \textit{slow invariant manifold} of Van der Pol \textit{singularly perturbed system} \eqref{eq59}. According to Eq. \eqref{eq22} and since $g'(x) = 1$, this \textit{flow curvature manifold} reads:

\begin{equation}
\label{eq61}
\phi ( x, y, \varepsilon ) = \ddot{x}\dot{y} + \dot{x}^2.
\end{equation}

On Fig. 1, we have represented, the \textit{trajectory curve}, integral of Van der Pol \textit{singularly perturbed system} \eqref{eq59}, i.e. the \textit{limit cycle} (in red), the \textit{critical manifold} $y - F(x) = 0$ with $F(x) = x^3 / 3 - x$, i.e. the zero order approximation in $\varepsilon$ of the \textit{slow invariant manifold} equation of this system (in green) and the roots of the equation $f(x) = x^2 - 1 = 0$ (in dot dashed black).

\begin{figure}[htbp]
\includegraphics[width=12cm,height=12cm]{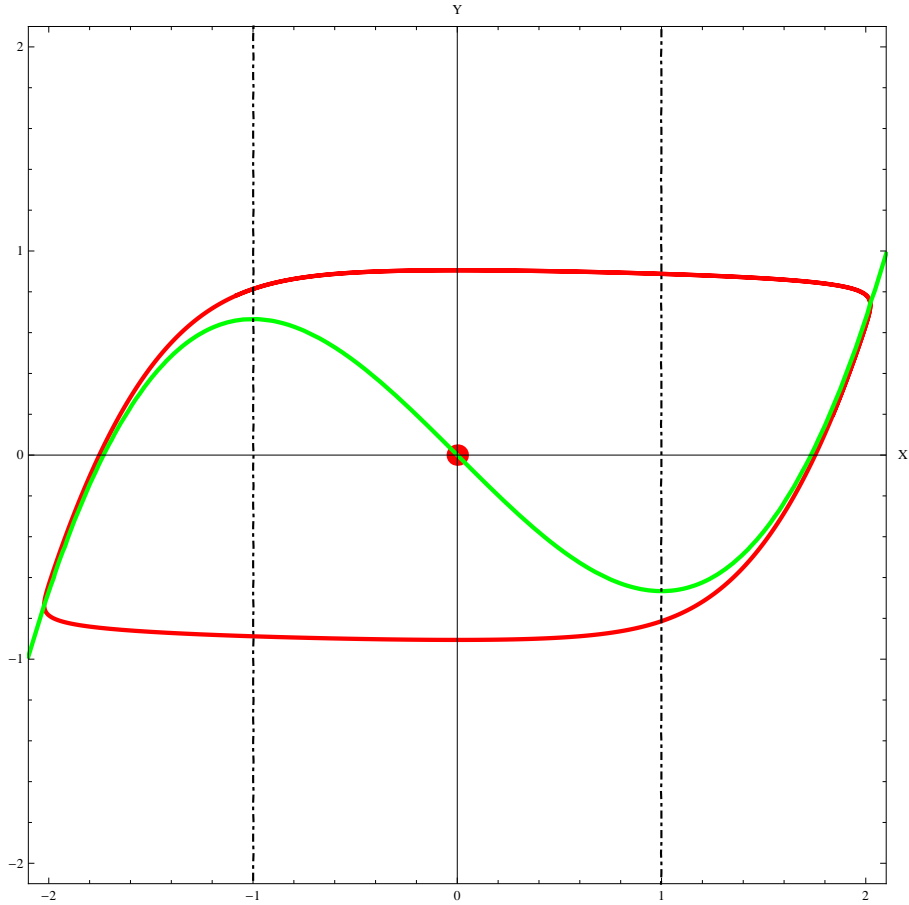}
\caption{Limit cycle and critical manifold of Van der Pol system \eqref{eq59}.}
\label{fig1}
\end{figure}

Due to the symmetries of the Van der Pol system \eqref{eq59}, let's focus on the right half part of the $xy$-plane of Fig. 1. It is easy to verify on the one hand that above the \textit{critical manifold} (in green), $y - F(x) > 0$ and, on the other hand that the \textit{trajectory curve}, i.e. the \textit{slow part} of the \textit{limit cycle} is below the \textit{critical manifold}. It follows that $\dot{x} < 0$ in the $\varepsilon$-vicinity of the \textit{slow invariant manifold} of Van der Pol system \eqref{eq59}. Moreover, since $x > 0$ in this part of the phase plane, we have $\dot{y} = - x < 0$ because $g'(x) =  1 \geqslant 0$. Thus both $x(t)$ and $y(t)$ decrease in this domain as previously stated in Sect. 5. Moreover, according to assumption IV, $F(x) = x^3 /3 - x$ is monotone increasing for $x \geqslant \sqrt{3}$. So $-F(x)$ is monotone decreasing and it has been stated by Lefschetz \cite{Lefschetz} that ``$y(t)$ decreases to the right of the $y$-axis'' which implies that $\dot{y} < 0$. This leads to $\dot{x}$ decreases in the $\varepsilon$-vicinity of the \textit{slow invariant manifold} and so $\ddot{x} < 0$. Let's notice that this result is intuitive because the point $M(x(t),y(t))$ of the \textit{trajectory curve}, integral of the Van der Pol system \eqref{eq59}, is then moving on the \textit{slow part} of the \textit{limit cycle}. From all these considerations, we deduce that $\ddot{x}\dot{y} > 0$. We thus prove that the \textit{flow curvature manifold} \eqref{eq61} that is to say, a first order approximation in $\varepsilon$ of the \textit{slow invariant manifold} of Van der Pol \textit{singularly perturbed system} \eqref{eq59} is positive. Such result could have been directly deduced from Proposition 4 since for Van der Pol system \eqref{eq59}, we have $g'(x) =  1 \geqslant 0$.

According to Eq. \eqref{eq27} and since $g'(x)=1$, $g''(x) = 0$, the time derivative of the \textit{slow invariant manifold} reads for the Van der Pol system \eqref{eq59}:

\begin{equation}
\label{eq62}
\dfrac{d\phi }{dt} = \dddot{x}\dot{y} + \ddot{x}\dot{x}
\end{equation}

\newpage

From the Van der Pol system \eqref{eq59} we deduce that:

\begin{equation}
\label{eq63}
\ddot {x} = \dfrac{1}{\varepsilon} \left[ \dot{y} - \left( x^2 -  1 \right) \dot{x} \right] = \dfrac{1}{\varepsilon} \left[ - x - \left( x^2 -  1 \right) \dfrac{1}{\varepsilon} \left[ y - F \left( x \right) \right] \right]. \\
\end{equation}

According to assumption IV, $F(x) = x^3 /3 - x$ is monotone increasing for $x \geqslant \sqrt{3}$. So $-F(x)$ is monotone decreasing and it has been stated by Lefschetz \cite{Lefschetz} that: `` $x(t)$ decreases below the \textit{critical manifold} and $y(t)$ decreases to the right of the $y$ axis'', it follows that $\ddot{x}$ decreases and so, $\dddot{x} < 0$. As a consequence $d\phi / dt >0$. Thus, since we have established that $dE /dt < 0$ and $d\phi / dt >0$, in the the $\varepsilon$-vicinity of the \textit{slow invariant manifold} of Van der Pol system \eqref{eq59}, Minorsky's statement is proved for such system.

\subsection{Generalized Li\'{e}nard \textit{singularly perturbed system}}

According to Llibre \textit{et al.} \cite{LlibreMereu}, an example of generalized Li\'{e}nard system can be written as follows:

\begin{equation}
\label{eq64}
\begin{aligned}
\varepsilon \dot {x} & = y - \left( \dfrac{x^5}{5} + \dfrac{x^3}{3} - x \right), \\
            \dot {y} & = - \left( \dfrac{x^3}{3} + x \right).
\end{aligned}
\end{equation}
Thus we have: $F(x) = \dfrac{x^5}{5} + \dfrac{x^3}{3} - x$, $f(x) = F'(x) = x^4 + x^2 - 1$, $g(x) = \dfrac{x^3}{3} + x$, $g'(x) = x^2 + 1$ and $G(x) = \dfrac{x^4}{12} + \dfrac{x^2}{2} + C$ where we can take $C = 0$.\\

According to \eqref{eq35} we have: $H = x^2 \left( \dfrac{x^4}{18} + \dfrac{x^2}{2} + 1 \right) \leqslant 0$, and according to \eqref{eq57}, $\dfrac{dH}{dt} = -4x \left( \dfrac{x^4}{12} + \dfrac{x^2}{2} \right) \dot{x} \geqslant 0$, because it has been stated that $\dot{x} < 0$. From \eqref{eq31}, it follows that:

\begin{equation}
\label{eq65}
\dfrac{dE}{dt} = -  \left(x^4 + x^2 - 1 \right) \dot{x}^2.
\end{equation}

The function $x^4 + x^2 - 1 \geqslant 0$ for $x \in ] - \infty, - \alpha] \bigcup [+\alpha, + \infty [$ with $\alpha = \sqrt{\dfrac{\sqrt{5}-1}{2}}$ and so, $dE/dt < 0$ within this interval which contains the \textit{flow curvature manifold} that is to say, a first order approximation in $\varepsilon$ of the \textit{slow invariant manifold} of generalized Li\'{e}nard \textit{singularly perturbed system} \eqref{eq64}. According to Eq. \eqref{eq22} and since $g'(x) = x^2 + 1$, this \textit{flow curvature manifold} reads:

\begin{equation}
\label{eq66}
\phi ( x, y, \varepsilon ) = \ddot{x}\dot{y} + \left( x^2 + 1 \right) \dot{x}^2.
\end{equation}

Using the same considerations as previously, it is easy to prove that the \textit{flow curvature manifold} \eqref{eq66} that is to say, a first order approximation in $\varepsilon$ of the \textit{slow invariant manifold} of generalized Li\'{e}nard \textit{singularly perturbed system} \eqref{eq64} is positive. Such result could have been directly deduced from Proposition 4 since for generalized Li\'{e}nard system \eqref{eq64}, we have $g'(x) = x^2 + 1 > 0$.

According to Eq. \eqref{eq27}, the time derivative of the \textit{slow invariant manifold} reads for generalized Li\'{e}nard system \eqref{eq64}:

\begin{equation}
\label{eq67}
\dfrac{d\phi }{dt} = \dddot{x}\dot{y} + \dfrac{d}{dt} \left( g'(x) \dot{x} \right)
\end{equation}

Thus from Proposition 5, i.e. since $g'(x)= x^2 + 1 > 0$ and by using the same considerations as previously, it is easy to prove that $d\phi / dt >0$. Therefore since we have established that $dE /dt < 0$ and $d\phi / dt >0$, in the the $\varepsilon$-vicinity of the \textit{slow invariant manifold} of generalized Li\'{e}nard system \eqref{eq64}, Minorsky's statement is proved for such system.

\section{Conclusion}

In this work, by using the \textit{Flow Curvature Method}, we have stated that in the $\varepsilon$-vicinity of the \textit{slow invariant manifold} of generalized Li\'{e}nard systems, the \textit{curvature of trajectory curve} increases while the \textit{energy} of such systems decreases. Hence we proved Minorsky's statement for the generalized Li\'{e}nard systems dating from half a century. Moreover we established a relationship between \textit{curvature} and \textit{energy} for such systems that he didn't provide. Some perspectives to be given to this work should be on the one hand analyze how \textit{curvature} and \textit{energy} could be related to the number of \textit{limit cycles} of such planar \textit{singularly dynamical systems}. On the other hand, it should be interesting to investigate if these results could be extended to higher dimensional \textit{singularly dynamical systems}.

It might also be interesting to study the relation between energy considerations and entropy concepts that have been used in the context of \textit{slow manifold} computation, see e.g. \cite{Leb04,Leb10}, where it has been demonstrated that minimum entropy production and minimum curvature are connected. This seems to be a conceptual analogy of classical thermodynamics with entropy characterizing the degree of energy dissipation.

\section{Appendix}

Starting from \eqref{eq8} we have the following equation:

\[
\e \ddot{x} + f\left( x \right) \dot{x} + g\left( x \right) = 0.
\]
By multiplying this equation by $\dot{x}(t)$ and by considering that $f(x) = F'(x)$, $g(x) = G'(x)$ we obtain:

\[
\e \dot{x}\ddot{x} + \left( F'(x)\dot{x} \right) \dot{x} + \left( G'(x)\dot{x} \right) = 0.
\]
Since $\e \ddot{x} = \dot{y} - F'(x)\dot{x}$ we have:

\[
\dot{x}\dot{y} +  G'(x)\dot{x} = 0.
\]
But $\e \dot{x} = y - F(x)$, and we find:

\[
y\dot{y} +  \e G'(x)\dot{x} = F(x)\dot{y}.
\]
Finally we obtain:

\[
\dfrac{d}{dt} \left( \dfrac{y^2}{2} +  \e G(x) \right) = F(x)\dot{y}.
\]
Then starting from $y = \e \dot{x} + F(x)$ we find that:

\[
\e \dfrac{d}{dt} \left( \e \dfrac{\dot{x}^2}{2} + G(x) \right) + \dfrac{d}{dt} \left( \e F(x) \dot{x} +\dfrac{F^2(x)}{2} \right) = F(x)\dot{y}.
\]
After simplifications we obtain the following equation:

\[
\dfrac{d}{dt} \left( \e \dfrac{\dot{x}^2}{2} + G(x) \right) = - f(x)\dot{x}^2,
\]

which is identical to Eq. \eqref{eq30}.

\section*{Acknowledgments}

The second author is supported by the Klaus-Tschira Foundation (Germany), grant 00.003.2019. The third author is partially supported the Ministerio de Ciencia, Innovaci\'on y Universidades, Agencia Estatal de Investigaci\'on grants MTM2016-77278-P (FEDER) and PID2019-104658GB-I00 (FEDER), the Ag\`encia de Gesti\'o d'Ajuts Universitaris i de Recerca grant 2017SGR1617, and the H2020 European Research Council grant MSCA-RISE-2017-777911.

\end{document}